\documentclass[12pt]{amsart}
\usepackage{amssymb,amsfonts,amsthm,mathtools,float,color,bm}
\usepackage{caption,subcaption,tikz,gensymb,tikz-3dplot, ytableau, hyperref}

\usetikzlibrary{positioning}
\usetikzlibrary{shapes}
\tikzset{main node/.style={ellipse,draw,minimum size=1cm,inner sep=1pt},
            }
            \tikzstyle{every node}=[draw, ellipse, minimum size=1cm,align=center]

\usepackage[hmargin=1in,vmargin=1in]{geometry}

\DeclareMathOperator{\Str}{Str}
\DeclareMathOperator{\Bx}{Box}
\DeclareMathOperator{\SCD}{SCD}
\DeclareMathOperator{\Mid}{Mid}

      \theoremstyle{plain}
      \newtheorem{theorem}{Theorem}[section]
      \newtheorem{lemma}[theorem]{Lemma}
      \newtheorem{corollary}[theorem]{Corollary}
      \newtheorem{proposition}[theorem]{Proposition}
      \theoremstyle{definition}
      
      \theoremstyle{remark}
      \newtheorem{remark}[theorem]{Remark}
      
      \theoremstyle{plain}
      \newtheorem{conjecture}{Conjecture}[section]
      \newtheorem{question}[conjecture]{Question}
      \newtheorem{problem}[conjecture]{Problem}

\def\Z{\mathbb{Z}}
\def\B{\mathcal{B}}
\def\redx{*(red) x}
\newcommand{\tbl}[1]{\textcolor{blue}{}}
\newcommand{\tgr}[1]{\textcolor{green}{}}

\DeclareMathOperator{\Sym}{Sym}

\DeclareMathOperator{\wt}{wt}

\DeclareMathOperator{\im}{Im}

\def\M{\mathcal{M}}

\def\multiset#1#2{\ensuremath{\left(\kern-.3em\left(\genfrac{}{}{0pt}{}{#1}{#2}\right)\kern-.3em\right)}}

\newcommand{\gauss}[2]{\genfrac{[}{]}{0pt}{}{#1}{#2}}

\newcommand{\N}{\mathbb{N}}
\newcommand{\C}{\mathbb{C}}
\newcommand{\R}{\mathbb{R}}

\author{Robert Dorward}

\begin{document}

\email{dorward@google.com}
\address{Google LLC\\
1600 Amphitheatre Pkwy\\
Mountain View, California, USA 94043}
\subjclass[2020]{05A05, 05A15, 05A17, 05E10, 17B10, 20C33}
\keywords{Minuscule lattice, $L(m,n)$, $M(n)$, strict partition, partition with distinct parts,  pattern,  standard Young tableaux, skew standard Young Tableaux, symmetric chain decomposition, pattern avoidance, Bruhat order, parabolic subgroup, partition in a box, Weyl group, crystal graph, crystal base, Gaussian polynomial, Lie algebra}

\title[The number of symmetric chain decompositions]{On the number of symmetric chain decompositions of the minuscule lattices $L(m,n)$ and $M(n)$}

\begin{abstract}
    We study the problem of enumerating symmetric chain decompositions (SCDs) of the minuscule lattices $L(m,n)$ of partitions in an $m$ by $n$ box and $M(n)$ of partitions into distinct parts at most $n$. We shift the focus from constructing a single SCD to analyzing the global structure of the set of SCDs. Let $\#\SCD(P)$ be the number of symmetric chain decompositions of $P$. We give an explicit formula for $\#\SCD(L(2,n))$ based on inversion sets of permutations and conjecture that for fixed $m>1$ both $\#\SCD(L(m,n))$ and $\#\SCD(M(n))$ grow super-exponentially. These conjectures are supported by data produced by AlphaEvolve, an evolutionary coding agent from Google DeepMind, and are in the same vein as a recent paper of Tomon on the growth rate of $\#\SCD$ for the Boolean lattice and hypergrid. We make connections with crystal bases and show that the Lusztig involution (evacuation) extends to an involution on SCDs, which we use to show that $\#\SCD(M(n))$ is even for $n>2$. We use skew tableaux sequences, which are equivalent to SCDs, and describe a potential path forward for finding SCDs through a notion of tableaux avoidance. We discuss implications of the conjectures for the problem of computing plethysm coefficients and discuss connections to physics and geometry. We end with a list of conjectures, questions and open problems.
\end{abstract}
\maketitle

\section{Introduction}

Let $M(n)$ be the poset of strict integer partitions with max part at most $n$ ordered by containment of Young diagrams. $M(n)$ is known to be a graded poset with rank function $\rho(\lambda) = |\lambda|$ and largest rank $\rho(M(n)) = \binom{n+1}{2}.$ Let $L(m,n)$ be the poset of integer partitions that fit inside an $m$ by $n$ rectangle, ordered by inclusion of Young diagrams. $L(m,n)$ is also known to be a graded poset with rank function $\rho(\lambda) = |\lambda|$ and largest rank $\rho(L(m,n)) = mn.$ A {\it symmetric chain decomposition} (SCD) of a finite graded poset $P$ is a partition of $P$ into subsets of totally ordered elements called {\it chains}, where each chain is symmetric and saturated. Given a chain $C$, let $\mu$ be the smallest element of $C$ and $\lambda$ be the largest element of $C$. A chain is {\it symmetric} if $\rho(\mu) +\rho(\lambda) = \rho(P)$. A chain is {\it saturated} if $|C| = \rho(\lambda)-\rho(\mu)+1$.  Stanley \cite{stanley1980weyl} raised the following question: 
\begin{question}\label{stanquest}
Do $L(m,n)$ and $M(n)$ admit an SCD?
\end{question}
This problem remains open, with a long history of partial progress for $L(m,n)$, though very little progress for $M(n)$. Given both $L(m,n)$ and $M(n)$ are graded, we can create their corresponding rank-generating functions, which are known to be $$F(L(m,n),q)  =  \sum_{\lambda \in L(m,n)} q^{|\lambda|} = \gauss{m+n}{n}_q $$
and 
$$F(M(n),q) = \sum_{\lambda \in M(n)} q^{|\lambda|} = \prod_{i=1}^n (1+q^i),$$
where we define q-analogs $[k]_q := 1 + q + q^2+\dots + q^{k-1}$, $[k]_q! := [k]_q[k-1]_q\dots[2]_q[1]_q$ and 
$$\gauss{k}{\ell}_q := \frac{[k]_q!}{[\ell]_q![k-\ell]_q!}.$$

A graded poset $P$ with a degree $\rho(P)$ rank-generating function $\sum_{i=0}^{\rho(P)} a_iq^i$ is {\it rank-symmetric} if $a_i = a_{\rho(P)-i}$ for all $0\leq i\leq \rho(P)$. A graded poset is {\it rank-unimodal} if there exists some index $j$ such that the pairwise inequalities hold: $a_0\leq a_1\leq \dots \leq a_{j-1} \leq a_j \geq a_{j+1} \geq a_{j+2} \geq \dots\geq a_{\rho(P)}$. A poset is {\it Sperner} if the size of the largest antichain is bounded by the size of the largest rank. A poset is {\it $k$-Sperner} if no union of $k$ antichains is larger than the union of the $k$ largest rank levels and {\it strongly Sperner} if it is $k$-Sperner for $1\leq k \leq \rho(P)+1$. A poset is called {\it Peck} if it is rank-symmetric, rank-unimodal and strongly Sperner.
\begin{theorem}{\cite{stanley1980weyl}} $L(m,n)$ and $M(n)$ are Peck.
\end{theorem}
The existence of a symmetric chain decomposition of a poset implies that it is Peck. However, Stanley used techniques from algebraic geometry to prove his theorem for a broader class of parabolic quotient posets $W^J$ corresponding to parabolic subgroups of Weyl groups under the Bruhat order, of which $M(n)$ and $L(m,n)$ are examples, and asked the equivalent of Question \ref{stanquest} for all of them.

Shortly after Stanley's proof, Proctor reinterpreted Stanley's proof in the language of representations of semisimple Lie algebras in a series of papers \cite{proctor1982representations, proctor1984bruhat, proctor1986dynkin, proctor1999dynkin, proctor1999minuscule}. He showed that a poset is Peck if and only if it ``carries a representation of $\mathfrak{sl_2}(\C)$." He showed that $W^J$ is a distributive lattice in the case where the corresponding representation has a highest weight that is minuscule and dominant, which we will discuss in Section \ref{crystal}. These are called {\it minuscule lattices}. Both $L(m,n)$ and $M(n)$ are minuscule lattices. He classified the minuscule lattices for Cartan types ADE, which was generalized by Stembridge \cite{stembridge2001minuscule} to the general case. Lastly, he generalized beyond semisimple Lie algebras to Kac-Moody algebras by generalizing the minuscule lattices to $d$-complete posets. 

There has been much partial progress on constructing SCDs for $L(m,n)$. Fixing $m\leq 4$, various constructions are given in \cite{riess1978zwei, lindstrom1980partition, west1979symmetric, greene1990greedy, dhand2014tropical, david2017geometric, wen2004computer, xin2021explicit, coggins2024visual, orellana2024quasi, gutierrez2024towards} and most recently for $m=5$ in 2022 and $m=6$ in 2026 by Wen \cite{wen2022symmetric, wen2026symmetric}. Guti\'errez \cite{gutierrez2024towards} recently gave a thorough literature review of the history of symmetric chain decompositions for $L(m,n)$ and we refer to that paper for more details on the history. We should also mention O'hara's \cite{o1990unimodality} celebrated constructive proof of the rank-unimodality of $F(L(m,n),q)$ by finding a symmetric chain decomposition for $L'(m,n)$, a related poset with much denser covering relations than $L(m,n)$.

In contrast, not much progress has been made on SCDs for $M(n)$, but it does appear in the literature in other contexts. Stanley's paper proved a conjecture made by Lindstr\"om \cite{lindstrom1970conjecture} that $M(n)$ was Sperner. In \cite{stanley1980weyl} Stanley claims that Hughes showed that $M(n)$ was rank-symmetric and rank-unimodal in \cite{hughes1977lie}. Graph-theoretic properties of the Hasse diagram of $M(n)$ were investigated in \cite{savage2003existence, tasoulas2024hamiltonian, merris2005lattice}. The poset $M(n)$ was shown to be related to permutation and set partition pattern avoidance in \cite{goyt2009set, dahlberg2016set, davis2018pattern, campbell2018restricted}. The poset $M(n)$ was studied from a computational complexity point of view in \cite{kubo2025partially}.

Although the problem of existence of SCDs for $L(m,n)$ and $M(n)$ is still open, when computing examples, we see that they appear to be extremely common. For a poset $P$, we let $\#\SCD(P)$ be the number of distinct SCDs of $P$. It's easy to see
\begin{proposition}
    $\#\SCD(L(1,n)) = 1$.
\end{proposition}

Based on data discussed in Section \ref{computation}, computed by AlphaEvolve \cite{novikov2025alphaevolve} from Google DeepMind, we make the following conjectures.

\begin{conjecture}\label{lconj}
For fixed $m>1$, for all fixed $k\in\N$, we have 
$$\lim_{n\rightarrow \infty} \frac{\#\SCD(L(m,n))}{k^n} = \infty.$$
\end{conjecture}

\begin{conjecture}\label{mconj}
For all fixed $k\in\N$, we have 
$$\lim_{n\rightarrow \infty} \frac{\#\SCD(M(n))}{k^n} = \infty.$$
\end{conjecture}

Recently Tomon \cite{tomon2025number} explored the equivalent question of enumerating SCDs for the Boolean lattice and the hypergrid, but there appears to be no literature on $\#\SCD$ for $M(n)$ and $L(m,n)$. Tomon gives exact asymptotics, but his proof relies on the ``LYM" property, which $M(n)$ and $L(m,n)$ lack. In addition, counting maximal chains in these posets is well-studied, yet there is little literature on the corresponding question for full SCDs. England \cite{englandthesis2026} studied algorithms for exact enumerations of $\#\SCD(L(m,n))$ and we prove a conjecture from that thesis in Section \ref{l2nsec}.

The overall goal of this paper is to shift the research focus from the construction of individual SCDs to the analysis of the global structure of the set of all possible SCDs on a given poset. In the specific case of $L(2,n)$, we demonstrate that this set possesses a rich structure governed by inversion sets of permutations. For $M(n)$, we utilize the global properties of the poset, viewed as a minuscule crystal, to uncover structure within its set of SCDs. Our introduction of tableaux avoidance for skew tableaux sequences provides a potential mechanism for identifying patterns in this global structure that we verify for small cases. Furthermore, our analysis of the methods employed by AlphaEvolve suggests that individual SCDs arise from complex global constraints rather than simple local rules, as evidenced by the system's reliance on heuristic approximations of Proctor’s raising and lowering operators and global crystal symmetries. This suggests that an individual SCD can be viewed as an arbitrary choice of ``routing" within the global structure of the poset.

We will now give a brief overview of the structure of the paper. In Section \ref{l2nsec} we give an exact enumeration of $\#\SCD(L(2,n))$. Next, in Section \ref{box}, we discuss {\it skew tableaux sequences}, which are in bijection with symmetric chain decompositions, and also introduce a notion of tableaux avoidance that might help characterize these. In Section \ref{crystal}, we connect the work with crystal bases and representation theory of Lie algebras and show how Lusztig's involution (evacuation) induces an involution on skew tableaux sequences, and therefore on symmetric chain decompositions. We use this to show that $\#\SCD(M(n))$ is even for $n>2$. In Section \ref{computation}, we give some computational evidence of our main conjectures, generated by AlphaEvolve and discuss some computational approaches it took. The exact enumeration of AlphaEvolve was limited by computer resource constraints and the main goal of the section is to look at how AlphaEvolve explored the space of algorithms.  Lastly, we end with some open questions and problems, including a discussion of related parts of geometry and physics as well as connections to the Langlands program. We enumerate all SCDs of $M(3), M(4), M(5)$ in an appendix.

\subsubsection*{Acknowledgments}
The author would like to thank Anne Schilling and Bruce Sagan for helpful correspondence, discussion and advice. The author would especially like to thank Anne Schilling for informing him of some of the results of Nicolas England's thesis. The author would also like to thank Robert Proctor and Richard Stanley for correspondence on $M(n)$.

\subsubsection*{Statement on AI use}
All proofs were generated by the author unless stated otherwise. Gemini 3.1 Pro and Gemini 3.5 Flash were used for literature searches to find references and many of the connections in Section \ref{questions}. Many numerical computations and computer simulations were performed by AlphaEvolve. The author takes full responsibility for any incorrect statements in the paper.

\section{Enumerating $\#\SCD(L(2,n))$}\label{l2nsec}
We say that a partition $\lambda$ fits inside a $m$ by $n$ rectangle (or box) if $\lambda_1 \leq n$ and $\ell(\lambda) \leq m$. Recall that $L(m,n)$ is the poset of partitions that fit in an $m$ by $n$ rectangle, ordered by inclusion of Young diagrams. We refer to \cite{sagan2001symmetric, fulton1997young} for background on partitions and Young tableaux and to Section 3 of \cite{stanley2011enumerative} for background on posets. We will always consider Young diagrams using English notation. 

We will begin by giving an exact formula for $\#\SCD(L(2,n))$. This result appears as a conjecture in \cite{englandthesis2026}.

\begin{theorem}\label{l2n}
$$\#\SCD(L(2,n)) = \left\lfloor\frac{n+2}{2}\right\rfloor!.$$
\end{theorem}
\begin{proof} Let $\Mid(n)$ be the set of partitions of $n$ into at most two parts, which is the set of partitions in the middle rank of $L(2,n)$. Let $S_{\Mid(n)}$ be the set of permutations of $\{1,2,\dots,\#\Mid(n)\}$. We will show that each SCD of $L(2,n)$ is in bijection with an element of $S_{\Mid(n)}$, which clearly has $\#S_{\Mid(n)} = \left\lfloor\frac{n+2}{2}\right\rfloor!$.

Let $\xi:\SCD(L(2,n))\longrightarrow S_{\Mid(n)}$ be defined by mapping each chain to the element of the $\Mid(n)$ that it contains, ordered by chain length from shortest to longest. The difference in sizes between the ranks of $L(2,n)$ is either zero or one, depending on whether the rank is even or odd. This shows $\xi$ is well-defined. We will show that $\xi$ defines a bijection by giving its inverse. 

To define $\xi^{-1}$, we first consider labeling the elements of $\Mid(n)$ in lexicographical order to give an element of $S_{\Mid(n)}$. So for example in $L(2,6)$, then $\Mid(6) = \{33, 42, 51, 6\}$ and  we label $33$ with $1$, $42$ with $2$, etc, which is an element of $S_{\Mid(n)}$. We will define our chains to have reflectional symmetry about the middle rank. So, we will describe the construction for ranks $n$ to $2n$ and then the constructions from rank $0$ to $n$ should be clear. So consider an element $\lambda\in \Mid(n)$ which corresponds to index $k$ in our permutation, which is mapped by $\xi^{-1}$ to some chain $C_k \in L(2,n)$, the $k$th shortest chain in our SCD. Then we can write $\lambda= (\lambda_1, \lambda_2)$. We will define an $L$ move to be $\lambda \rightarrow (\lambda_1, \lambda_2+1)$ and an $R$ move to be $\lambda \rightarrow (\lambda_1+1, \lambda_2)$. Then a starting $\lambda$ and a sequence of moves in $\{R, L\}$ defines a symmetric chain, assuming that the number of $R$ is at most $n-\lambda_1$, the number of $L$ is at most $n-\lambda_2$ and the parity of the total number of $L$ and $R$ is the parity of $n$. We will call this the $RL$ word of the chain and denote it $w$.  

First, define subwords for $i<j$ 
$$I(i,j) = \begin{cases}
RL & \text{if $i$ and $j$ form an inversion} \\
LL & \text{otherwise.}
    \end{cases}$$
    and $I(i,i)$ is defined to be empty.

Then our $RL$ word $w$ for $C_k$ is defined to be 
$$w = \begin{cases} 
\prod\limits_{i=1}^k I(i,k) & \text{if $n$ is even},\\
L\left( \prod\limits_{i=1}^k I(i,k) \right) & \text{if $n$ is odd},
\end{cases}$$
where we take multiplication to mean concatenation in this context.

As an example, recall in $L(2,6)$ we have $\Mid(6) = \{33, 42, 51, 6\}$. Then we can define an ordering $51, 33, 6, 42$, which corresponds to $3142\in S_4$. Then the $RL$ word of $51$ is empty, $33$ is $RL$, $6$ is $LLLL$ and $42$ is $RLLLRL.$ Figure \ref{l26scd} shows the corresponding SCD.

The number of inversions of $(\lambda_1, \lambda_2)$ is at most $n-\lambda_1$, so $\xi^{-1}$ creates well-defined chains. Clearly the chains are symmetric and saturated by construction. 

We now show that they partition $L(2,n)$. First, let us consider the ``hooks" of $L(2,n)$, which are chains starting at partition $(k,k)$, ending at $(n-k, n-k)$ with $RL$ word $L^{n-2k}$. We can see an example in Figure \ref{hooks}. Each inversion in our permutation corresponds to shifting the current chain to the next largest hook, while a non-inversion corresponds to staying on the current hook. When there is a inversion $I(i,k)$, that means the chain which started on the hook corresponding to element $i$ is ending, and so a chain containing a smaller element of $\Mid(n)$ must start covering that hook. For even $n$, the smallest hook is of length 1 and for odd $n$, length 3, which explains the extra $L$ in the $RL$ word for odd $n$. We can see by construction that every element on each hook is contained in exactly one chain.

Lastly, we must show these are inverses. It is clear that $\xi(\xi^{-1}(\sigma)) = \sigma$ by construction. Let $\xi(\{C_i\}) = \sigma$ and $\xi^{-1}(\sigma) = \{D_i\}$. Then the chains from each SCD that contain the same middle element are the same length. Notice that the chain taking the element $n$ cannot take any edges on any other hooks. Therefore, moving from the center upwards, it continues along the hook until it ends at some $(\lambda_1,\lambda_2)$ and if this is not the entire hook, the element $(\lambda_1, \lambda_2+1)$ must be in the same chain as $(\lambda_1-1, \lambda_2+1)$ and the chain containing these two elements must proceed along the outer hook eventually taking $(n,n)$. Since the lengths of the chains containing the middle element are the same in both SCDs, they must have the same corresponding $(\lambda_1,\lambda_2)$ where they end. We can see that removing the outer hook gives a poset isomorphic to $L(2,n-2)$ and therefore we can assume inductively that $\{C_i\}$ and $\{D_i\}$ are the same when restricted to these sets, since they correspond to the same permutation when restricted. Therefore, we have $\{C_i\} = \{D_i\}$.
\end{proof}

We remark that Theorem \ref{l2n} shows that Conjecture \ref{lconj} holds for $m=2$.

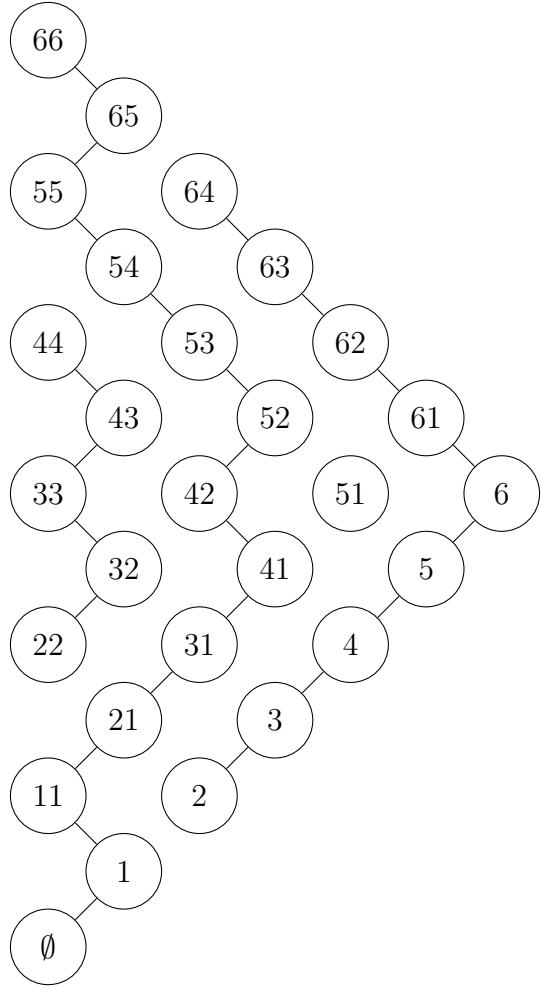
\begin{figure}[h]
\centering
\begin{tikzpicture}[scale=1]
  \node (empty) at (0,0) {$\emptyset$};
  \node (1) at (1,1) {$1$};
  \node (11) at (0,2) {$11$};
  \node (2) at (2,2) {$2$};
  \node (21) at (1,3) {$21$};
  \node (3) at (3,3) {$3$};
  \node (22) at (0,4) {$22$};
  \node (31) at (2,4) {$31$};
  \node (4) at (4,4) {$4$};
  \node (32) at (1,5) {$32$};
  \node (41) at (3,5) {$41$};
  \node (5) at (5,5) {$5$};
  \node (33) at (0,6) {$33$};
  \node (42) at (2,6) {$42$};
  \node (51) at (4,6) {$51$};
  \node (6) at (6,6) {$6$};
  \node (43) at (1,7) {$43$};
  \node (52) at (3,7) {$52$};
  \node (61) at (5,7) {$61$};
  \node (44) at (0,8) {$44$};
  \node (53) at (2,8) {$53$};
  \node (62) at (4,8) {$62$};
  \node (54) at (1,9) {$54$};
  \node (63) at (3,9) {$63$};
  \node (55) at (0,10) {$55$};
  \node (64) at (2,10) {$64$};
  \node (65) at (1,11) {$65$};
  \node (66) at (0,12) {$66$};
  \draw (empty) -- (1) -- (11) -- (21) -- (31) -- (41) -- (42) -- (52) -- (53) -- (54) -- (55) -- (65) -- (66);
  \draw (22) -- (32) -- (33) -- (43) -- (44);
  \draw (2) -- (3) -- (4) -- (5) -- (6) -- (61) -- (62) -- (63) -- (64);
\end{tikzpicture}
\caption{\label{l26scd} An SCD of $L(2,6)$ corresponding to the permutation $3142$.}
\end{figure}

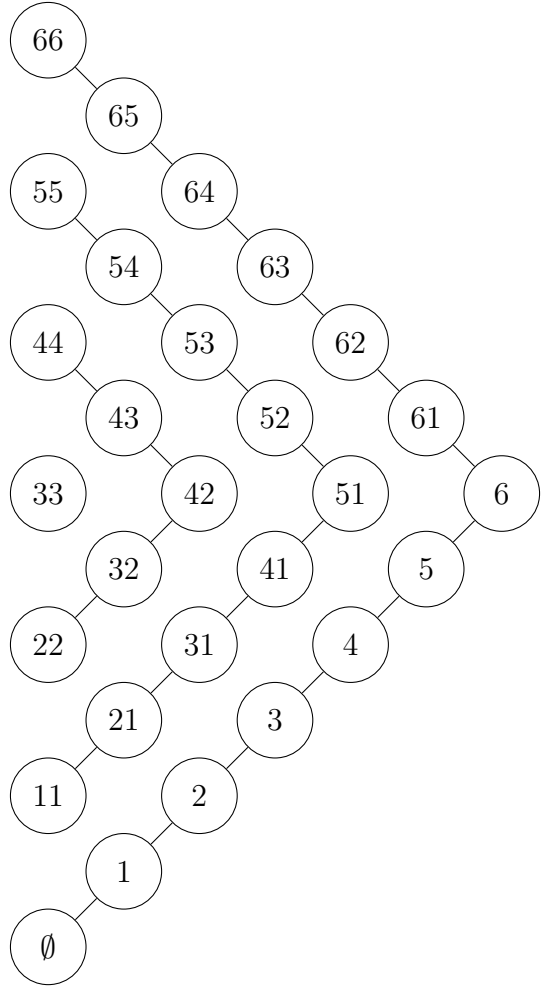
\begin{figure}[h]
\centering
\begin{tikzpicture}[scale=1]
  \node (empty) at (0,0) {$\emptyset$};
  \node (1) at (1,1) {$1$};
  \node (11) at (0,2) {$11$};
  \node (2) at (2,2) {$2$};
  \node (21) at (1,3) {$21$};
  \node (3) at (3,3) {$3$};
  \node (22) at (0,4) {$22$};
  \node (31) at (2,4) {$31$};
  \node (4) at (4,4) {$4$};
  \node (32) at (1,5) {$32$};
  \node (41) at (3,5) {$41$};
  \node (5) at (5,5) {$5$};
  \node (33) at (0,6) {$33$};
  \node (42) at (2,6) {$42$};
  \node (51) at (4,6) {$51$};
  \node (6) at (6,6) {$6$};
  \node (43) at (1,7) {$43$};
  \node (52) at (3,7) {$52$};
  \node (61) at (5,7) {$61$};
  \node (44) at (0,8) {$44$};
  \node (53) at (2,8) {$53$};
  \node (62) at (4,8) {$62$};
  \node (54) at (1,9) {$54$};
  \node (63) at (3,9) {$63$};
  \node (55) at (0,10) {$55$};
  \node (64) at (2,10) {$64$};
  \node (65) at (1,11) {$65$};
  \node (66) at (0,12) {$66$};
  \draw (empty) -- (1) -- (2) -- (3) -- (4) -- (5) -- (6) -- (61) -- (62) -- (63) -- (64) -- (65) -- (66);
  \draw (11) -- (21) -- (31) -- (41) -- (51) -- (52) -- (53) -- (54) -- (55);
  \draw (22) -- (32) -- (42) -- (43) -- (44);
\end{tikzpicture}
\caption{\label{hooks} The hooks of $L(2,6)$.}
\end{figure}

\section{Skew tableaux sequences}\label{box}
\subsection{Strict and box skew tableaux sequences}

Due to the rapid growth rate of $\#\SCD(L(m,n))$ and $\#\SCD(M(n))$, we will need a more convenient way to characterize SCDs. We do this by translating them to sequences of tableaux which we will call {\it skew tableaux sequences}. Analyzing the structure of the skew tableaux sequences that correspond to SCDs gives us insight into the global structure the SCDs must follow. For skew tableaux of shape $\lambda / \mu$, we will often need to know the shape of $\mu$ and so it will be drawn filled with red x's. The following correspondence between chains in Young's lattice and skew SYT is well known, see for example \cite{stanley1984number}.

Given a chain $C\subseteq M(n)$ or $C\subseteq L(m,n)$, with smallest element $\mu$ and largest element $\lambda$, we can associate it to a skew standard Young tableaux (skew SYT) of shape $\lambda / \mu $ which we will denote $\phi(C)$ and construct as follows. Given a skew SYT $T$ and $0\leq i \leq |\lambda / \mu|$, let $R(T,i)$ be the shape of the Young diagram formed by deleting all boxes filled with numbers greater than $i$. Then $R(T,0) = \mu$ and $R(T,|\lambda /\mu|) = \lambda$ and $R(T,i)$ is always a partition. If $T = \phi(C)$ for some chain $C$, then $R(T, i)$ is the $(i+1)$th element of $C$.

Fixing $m,n$, if both $\lambda$ and $\mu$ fit in an $m$ by $n$ box, then we will call $T$ a {\it box skew standard Young tableau} and in this case, it is equivalent to a chain in $L(m,n)$. We say that $T$ is a {\it strict skew standard Young tableaux} if $R(T,i)$ is a strict partition for all $0\leq i \leq |\lambda / \mu|.$ In this case, it is equivalent to a chain in $M(n)$.

\begin{lemma}
A skew SYT $T$ of shape $\lambda / \mu$ is strict if and only if the fillings of the antidiagonals of $T$ are increasing and both $\lambda$ and $\mu$ are strict.
\end{lemma}
\begin{proof}
Let $T$ be a strict skew SYT and consider the fillings of two boxes $a$ and $b$ where $a$ is one box above and to the right of $b$. Then we wish to show that $a < b$. Let $c$ be the filling of the box to the left of $a$ and above $b$, letting $c=0$ if the box is part of $\mu$. Because $T$ is a skew SYT, we have $c < a$ and $c < b$. We can see that $b < a$ would cause $R(T, b)$ to contain two parts of equal length, corresponding to rows ending in $b$ and $c$, which violates the assumption that $T$ is strict.

For the other direction, let $T$ be a skew SYT where fillings of the antidiagonals of $T$ are increasing. Suppose $R(T,i)= \omega$ is not strict for some $i$. Then we have consecutive parts $\omega_k$ and $\omega_{k+1}$ that are the same size. As in the previous case, we have fillings $c < b$, which are the right most boxes of $\omega_{k}$ and $\omega_{k+1}$, where $c=0$ if $\omega_k$ is a part of $\mu$. Because $\lambda$ is strict, it must be the case that there exists some $j > i$ where $R(T,j)$ has part $\omega_k + 1$ at index $k$. But this contradicts the assumption that the antidiagonals of $T$ are increasing.
\end{proof}

We should note that strict skew SYT are in bijection with shifted SYT from \cite{sagan1987shifted} by shifting row $i$ to the right $i-1$ boxes and that shifted SYT are another well-known formulation of $M(n)$.

Next, we will define a {\it strict skew tableaux sequence} $(T_i)$ to be a finite subset of strict SYT $(T_1, \dots, T_m)$ of shapes $(\lambda_1/\mu_1, \dots, \lambda_m/\mu_m)$ where $R(T_i, k) \neq R(T_j, \ell)$ for $1\leq i,j \leq m$ where $i\neq j$ and all $0\leq k \leq |\lambda_i / \mu_i|$ and $0\leq \ell \leq |\lambda_j / \mu_j|$. Let $\Str(n)$ be the set of strict skew tableaux sequences where all $\lambda_i, \mu_i \in M(n)$. Let $\SCD(M(n))$ be the (conjecturally nonempty) set of all SCDs of $M(n)$. Then, given a set of chains $\{C_1,\dots,C_m\}$, we define $$\phi_s:\SCD(M(n)) \longrightarrow \Str(n)$$ such that $$\phi_s(\{C_1,\dots,C_m\}) = \{\phi(C_1), \dots, \phi(C_m)\}.$$ We can immediately see that the image $\im(\phi_s) \neq \Str(n)$ because while $\Str(n)$ corresponds to subsets of chains that are saturated and disjoint, it is not necessarily the case that the pullback $\bigcup_i \phi_s^{-1}(T_i)$ is all of $M(n)$, nor that $\rho(\lambda_i) + \rho(\mu_i) = \binom{n+1}{2}$. We will call a tableaux sequence $(T_i) \in \Str(n)$ {\it strict admissible} if $(T_i)\in \im(\phi_s)$.

We can equivalently define $\Bx(m,n)$ to be the set of box skew tableaux sequences, for $\lambda$ that fit in an $m$ by $n$ box. Then we can also define $$\phi_b:\SCD(L(m,n)) \longrightarrow \Bx(m,n)$$ in a similar manner to $\phi_s$. We will call a tableaux sequence $(T_i) \in \Bx(n)$ {\it box admissible} if $(T_i)\in \im(\phi_s)$. 

As an example, Figure \ref{l26skew} shows the application of $\phi_b$ to the SCD in Figure \ref{l26scd}.

\begin{figure}[h]
\centering
\def\arraystretch{2.5}
\begin{tabular}{ll}
$
    \begin{ytableau}
1 & 3 & 4 & 5 & 7 & 11 \\
2 & 6 & 8 & 9 & 10 & 12 \\
\end{ytableau}$ & $\begin{ytableau}
\redx & \redx & 1 & 2 & 3 & 4  \\
5 & 6& 7 & 8\\
\end{ytableau}$ \\
$\begin{ytableau}
\redx & \redx & 1 & 3 \\
\redx & \redx & 2 & 4 \\
\end{ytableau}$ & $\begin{ytableau}
\redx & \redx & \redx & \redx & \redx\\
\redx \\
\end{ytableau}$
\end{tabular}
\caption{\label{l26skew} The box skew tableaux sequence corresponding to Figure \ref{l26scd}.}
\end{figure}

\begin{remark} In practice, the skew tableaux sequences we wish to consider will have a natural order which partially orders $T_i$ from smallest to largest based on $|\lambda|$, which justifies using the term ``sequence." In the proof of Theorem \ref{l2n} we took advantage of the fact that these sequences are {\it totally ordered} for $L(2,n)$. \end{remark}

\begin{remark}
    The tableaux sequences we consider are implicitly discussed for $L(m,n)$ in \cite{orellana2024quasi}, figures 1 and 3.
\end{remark}

\section{Crystals and algebraic structure}\label{crystal}
\subsection{Minuscule crystals}
For this section, it will be convenient to translate our problem into the world of crystal bases. The algebraic and crystal structure of the poset allows us to uncover constraints on the set of all SCDs of a minuscule lattice. For general background on crystal bases, we recommend \cite{bump2017crystal}. Our definitions will not be the most general, but only what is necessary for our purposes. The general correspondence between crystals and minuscule lattices is known (see for example \cite{green2007full, donnelly2018poset, strayer2018unified, dranowski2024heaps}), though the application to symmetric chain decompositions is new to the best of the author's knowledge.

Fix a finite semi-simple Lie algebra $\mathfrak{g}$ with root system $\Phi$, index set $I$, weight lattice $\Lambda$ and Weyl group $W$. A {\it Kashiwara crystal of finite type} is a nonempty set $\B$ together with maps
\begin{align*}
e_i,f_i:&\B \longrightarrow \mathcal{B}\cup \{0\} \\
\varepsilon_i,\varphi_i: &\B \longrightarrow \Z \\
\wt : &\B \longrightarrow \Lambda\\
\end{align*}
where $i\in I$ and $0\notin \B$ is an auxiliary element, satisfying the following axioms:

\begin{description}
  \item[A1] If $x,y\in \B$ then $e_i(x) = y$ if and only if $f_i(y) = x$. In this case, $\wt(y) = \wt(x) + \alpha_i$, $\varepsilon_i(y) = \varepsilon_i(x) - 1$ and $\varphi_i(y) = \varphi_i(x) + 1$, where $\alpha_i \in \Phi$. 
  \item[A2] We require that $\varphi_i(x) = \langle\wt(x),\alpha_i^{\vee}\rangle + \varepsilon_i(x)$ for all $x\in \B$ and $i\in I$, where $\alpha_i^{\vee}\in \Phi^{\vee}$ is the coroot of $\alpha_i$.
\end{description}

Every crystal base has a corresponding crystal graph, which is a directed graph with vertices in $\B$ and edges labeled by $i\in I$, where $x$ and $y$ have an edge labeled $i$ if and only if $f_i(x)=y$. We will often refer to crystal bases and their corresponding graphs interchangeably.

An element $\lambda\in\Lambda$ is called {\it minuscule} if $\langle\lambda, \alpha^{\vee}\rangle \in \{0,1,-1\}$ for every $\alpha^{\vee}\in\Phi^{\vee}$ and {\it dominant} if $\langle\lambda, \alpha^{\vee}\rangle \geq 0$ for all $\alpha^{\vee}$. The weights of the irreducible representation with highest dominant minuscule weight $\lambda$ consist of a single $W$-orbit. Therefore we can define a crystal $\mathcal{M}_\lambda$ where given $\mu$ in the $W$-orbit of $\lambda$, we have an element $v_\mu\in \mathcal{M}_\lambda$ where 
$$
f_i(v_\mu) =\begin{cases}
v_{\mu-\alpha_i} & \text{if $\langle\lambda, \alpha_i^{\vee}\rangle=1$},\\
0 & \text{otherwise,}
\end{cases}$$
and $$ e_i(v_\mu) =\begin{cases}
v_{\mu+\alpha_i} & \text{if $\langle\lambda, \alpha_i^{\vee}\rangle=-1$},\\
0 & \text{otherwise.}\\
\end{cases}
$$
We will be concerned with $\M_{\varpi_n}$ which is the minuscule crystal associated to the spin weight $\varpi_n$ associated to Cartan type $B_n$, as well as minuscule crystals of type $A_n$.

We can define $\M_{\varpi_n}$ as follows. First let $\bm{e}_i$ be $(0,0,\dots,0, 1, 0,\dots, 0)$ be the basis vector for $\R^n$ with a 1 in the $i$th position. Then the weights that can appear in these crystals of the form
$$\frac{1}{2}\sum_{i=1}^{n}\epsilon_i\bm{e}_i$$
where $\epsilon_i = \pm$. Then an element of $\M_{\varpi_n}$ can be uniquely described by $\epsilon_1\dots\epsilon_n$. For example $\epsilon_1\epsilon_2\epsilon_3 = ++-$ corresponds to $\frac{1}{2}(\bm{e}_1 +\bm{e}_2-\bm{e}_3)$.

\subsection{Proctor's representations}

As alluded to in the introduction, Proctor showed that $L(m,n)$ and $M(n)$ had the Peck property by showing they ``carried a representation of $\mathfrak{sl}_2(\C)$" which we will now rigorously define, following \cite{proctor1982representations}.

Let $P$ be a graded poset $$P = \bigcup_{i=0}^N P_i$$ where the $P_i$ are the ranks. We can associate to $P$ a graded complex vector space $$\tilde{P} = \bigoplus\limits_{i=0}^{N} \tilde{P_i},$$ where $\tilde{P_i}$ is the complex vector space freely generated by vectors $\tilde{a}$ corresponding to $a\in P_i$. A linear operator $X$ on $\tilde{P}$ is a {\it lowering operator} if $X\tilde{P_i}\subseteq\tilde{P}_{i-1}$ and a {\it raising operator} if $X\tilde{P_i}\subseteq\tilde{P}_{i+1}$. A raising operator
$$X\tilde{a} = \sum \Theta(a,b)\tilde{b}$$ is an {\it order raising operator} if $\Theta(a,b)\neq 0$ implies $b$ covers $a$. We can also define the linear operator $H\tilde{a} = (2i-n)\tilde{a}$ for $a\in P_i$. A poset $P$ {\it carries a representation of $\mathfrak{sl}_2(\C)$} if there exist a lowering operator $Y$ and an order raising operator $X$ on $\tilde{P}$ such that $[X,Y] = H$, where $[.,.]$ denotes the Lie bracket.
\begin{theorem}{\cite{proctor1982representations}}
    A graded poset is Peck if and only if it carries a representation of $\mathfrak{sl}_2(\C)$.
\end{theorem}

Proctor considers the special case which $L(m,n)$ and $M(n)$ fall into. A {\it uniquely modular poset} is a graded poset that satisfies
\begin{enumerate}
    \item Whenever two elements both cover a third element, there is a unique fourth element covering both of them,
    \item Whenever two elements both are covered by a third element, there is a unique fourth element that they both cover.
\end{enumerate}

A uniquely modular poset $P$ of length $N$ is {\it edge-labelable} if each covering relationship $a < d$ can be assigned a rational number $y(d,a)$ such that 
\begin{enumerate}
    \item If $d$ covers $a$ and $b$ and both $b$ and $a$ cover $c$, then $y(d,a) = y(b,c)$,
    \item If $a\in P_i$, then
    $$ \sum_{a\text{ covers } c} y(a,c) - \sum_{d\text{ covers } a} y(d,a) = 2i-N.$$
\end{enumerate}

Then we have \begin{proposition}{\cite{proctor1982representations}}
    Edge-labelable uniquely modular posets are Peck.
\end{proposition}

Both $L(m,n)$ and $M(n)$ are edge-labelable and uniquely modular.

We can think of $y(a,b)$ as a color of the edge. In \cite{proctor1982solution}, Proctor gives the following direct expressions for $y(a,b)$ in $L(m,n)$ and $M(n)$. In both cases, $b$ covers $a$ implies that there is a specific part $a_i$ such that $a_i = b_i-1$ and $a_j = b_j$ for all $j\neq i$. Then for $L(m,n)$, $$y(a,b) = (m-a_i+i-1)(n+a_i-i+1)$$
 and for $M(n)$
 $$y(a,b) =\begin{cases}
     \binom{n+1}{2} & \text{if $a_i=0$,}\\
     (n-a_i)(n+a_i+1) & \text{otherwise.}
 \end{cases}$$

As a note, Proctor's formula for $L(m,n)$ is different because flipped the roles of $n$ and $m$ and his partitions had parts in increasing order, so his $i$ is $n-i+1$ for us. Importantly, we can see that for $M(n)$, the weight or ``color" only depends on the specific $a_i$.
\subsection{Connecting crystals and Proctor's representations}

\begin{proposition}
The Hasse diagram of $M(n)$ is isomorphic to the crystal graph of $\mathcal{M}_{\varpi_n}$ and furthermore the coloring of $M(n)$ via Proctor's weights $y(a,b)$ are in bijection with the $e_i$ of the minuscule spin crystal.
\end{proposition}
\begin{proof}
    Given $\lambda\in M(n)$, define a sequence of pluses and minuses $\epsilon_1\dots \epsilon_n$ where $\epsilon_i = +$ if $n-i+1$ is a part of $\lambda$. Recall that every covering relation $\lambda \lhd \lambda'$ in $M(n)$ can be described by moving a part $\lambda_i$ to $\lambda_i+1 = \lambda'_i$, where $\lambda_i$ can potentially be a 0 padded to the end of $\lambda$. This corresponds to 
$$\psi(\lambda')-\psi(\lambda) =\begin{cases}
    \bm{e}_{n-\lambda_i}-\bm{e}_{n-\lambda_i+1} & \text { if } \lambda_i >0,\\
    \bm{e}_n & \text { if } \lambda_i =0.\\
\end{cases}$$ Note that these map precisely to the definition of $\alpha_i$, the simple roots of $B_n$ by adding an extra factor of $\frac{1}{2}$. Since we map $\lambda_i$ to $\alpha_i$, and Proctor's weights depend only on $\lambda_i$, they are in bijection with the $e_i$ of $\mathcal{M}_{\varpi_n}$.
\end{proof}

For $L(m,n)$, we can construct a corresponding minuscule crystal $\mathcal{M}_{\varpi_m}$ of Cartan type $A_{n+m-1}$ using the well-known lattice path construction (see \cite{stanley2011enumerative} Figure 1.15 for example) and then noting that the $e_i$ flip a 01 at indices $i(i+1)$ to a 10 and vice versa for $f_i$ in the lattice path. Similarly to the $M(n)$ case, the colors depend only on $a_i$ and $i$ and so do the corresponding crystal raising and lowering operators. However, we note that the mapping is not one to one to Proctor's weights as the color depends on $a_i-i$ instead.

\subsection{Lusztig's involution and tableaux sequences}
We now connect the tableaux sequences from Section \ref{box} to the crystals we are discussing. First, the columns of the skew tableaux have a nice interpretation in the crystal sense because the columns correspond to moving a specific part $\lambda_i$ to $\lambda_i+1$.
\begin{proposition}
    For strict tableaux sequences, the columns of $T$ correspond to the color/weight of the corresponding edge of the Hasse diagram of $M(n)$. Equivalently, the columns correspond to the raising/lowering operators of $\mathcal{M}_{\varpi_n}$.
\end{proposition}

Next, we discuss Lusztig's involution \cite{lusztig1990canonical, lenart2007combinatorics} (also known as the Sch\"utzenberger involution in type $A$ \cite{schutzenberger1972promotion} or shifted evacuation in type $B$ \cite{haiman1992dual}) and use it to give explicit involutions on admissible tableaux sequences. This involution is also often called {\it evacuation}.

Every Weyl group has a long element $w_0$ of order 2 such that there exists a permutation $I'$ of the index set $I$ such that $w_0(\alpha_i) = -\alpha_{i'}$. Using this we can define the {\it crystal involution} $S:\B \longrightarrow \B$ such that 
\begin{align*}
\wt(Sx) &= w_0\wt(x),\\
e_i(Sx) &= f_{i'}(x), \\
f_i(Sx) &= e_{i'}(x), \\
\varepsilon_i(Sx) &= \varphi_{i'}(x),\\
\varphi_i(Sx) &= \varepsilon_{i'}(x).\\
\end{align*}
This is the Lusztig involution and the two following notions of ``complement" that we give are realizations of this involution for $L(m,n)$ and $M(n)$. For $\lambda\in L(m,n)$, $S(\lambda) = \lambda^c$, where $\lambda^c$ is the complement of $\lambda$ within the $m$ by $n$ box, rotated 180 degrees. For $\lambda \in M(n)$, $S(\lambda)$ is the ``complement" with respect to the staircase partition $ST_n = (n,n-1,\dots,2,1)$. In other words, $S(\lambda)$ contains $k$ as a part if and only if $\lambda$ does not. We will denote this $\lambda^{c'}$.

We can define involutions $\eta_s$ and $\eta_b$ on admissible strict tableaux sequences and admissible box tableaux sequences respectively. Given a strict skew tableau $T$ of shape $\lambda/\mu$, define $\eta_s(T)$ as follows. We let $\eta_s(T)$ have shape $\mu^{c'} / \lambda^{c'}$. Let $c_i$ be the column of the filling of $i$ in $T$. Then we can form a column word of length $|\lambda/\mu|$. Fill $\eta_s(T)$ such that the column word is the reverse of $T$. For example if in $M(5)$ if $$T =\begin{ytableau}
\redx & \redx & \redx & \redx  &\redx \\
1&2&3&4\\
5 \\
\end{ytableau}$$ then $\lambda = (5,4,1)$, $\mu=(5)$, the column word is $12341$ and $$\eta_s(T) = \begin{ytableau}
\redx & \redx & \redx & 2 \\
\redx & \redx & 3\\
1& 4\\
5 \\
\end{ytableau}$$ with $\mu^{c'} = (4,3,2,1)$, $\lambda^{c'} = (3,2)$ and column word $14321$. We can define $\eta_s$ on a sequence of tableaux by applying it to each tableaux individually.

We can define $\eta_b$ similarly. Given a skew tableaux $T$ of shape $\lambda / \mu$ where $\lambda$ and $\mu$ fit in an $m$ by $n$ box, define $\eta_b$ as follows. $\eta_b(T)$ has shape $\mu^c / \lambda^c$. Let $r := r_1\dots r_{|\lambda / \mu|}$ be the reading word of $T$. Then fill $\eta_b(T)$ such that the reading word is the reverse complement $(|\lambda/\mu|+1-r_{|\lambda/\mu|})(|\lambda/\mu|+1-r_{|\lambda/\mu|-1})\dots(|\lambda/\mu|+1-r_1)$. For example if in $L(3,4)$ if $$T =\begin{ytableau}
\redx & \redx & \redx & 1 \\
2 & 3\\
4 & 5 \\
\end{ytableau}$$ then $\lambda = (4,2,2)$, $\mu=(3)$, the reading word is $45231$ and $$\eta_b(T) = \begin{ytableau}
\redx & \redx & 1 & 2 \\
\redx & \redx & 3 & 4\\
5 \\
\end{ytableau}$$ with $\lambda^{c} = (3,3,1)$, $\mu^{c} = (2,2)$ and reading word $53412$.

The fact that both $\eta_s$ and $\eta_b$ are involutions follows immediately from the observation that an involution on the corresponding crystal induces an involution on any set of chains. This, combined with the observation that $|\lambda^{c'}| = \binom{n+1}{2} - |\lambda|$ and $|\lambda^c| = mn-|\lambda|$, implies the following

\begin{theorem}
    If $(T_i)$ is an admissible strict tableaux sequence, then $(\eta_s(T_i))$ is an admissible strict tableaux sequence.
\end{theorem}
\begin{theorem}
    If $(T_i)$ is an admissible box tableaux sequence, then $(\eta_b(T_i))$ is an admissible box tableaux sequence.
\end{theorem}

We can immediately use this to prove the following. We remark that the proof was generated by Gemini 3.1 Pro, although it is not difficult to prove and the proof provided below was simplified by the author.

\begin{proposition}\label{nofixed}
For $n>2$, the map $\eta_s$ does not have a fixed point.
\end{proposition}
\begin{proof}
Any SCD of $M(n)$ must have a unique longest chain $C$. Notice that for $n>2$, $\lambda \neq \lambda^{c'}$.  If $\binom{n+1}{2}$ is even, there is a middle rank and therefore the element of $C$ in this middle rank is mapped to a different chain. If $\binom{n+1}{2}$ is odd, there are two middle ranks and we can consider $\lambda$ to be the element of $C$ at the lower of the two middle ranks. We can see that $\lambda^{c'}$ doesn't cover $\lambda$ and therefore must be on a different chain. 
\end{proof}

\begin{corollary}\label{coro}
For $n>2$, $\#\SCD(M(n))$ is even.
\end{corollary}

However, it is often the case that $(\eta_b(T_i)) = (T_i)$. In fact, for $L(2,n)$, we can see that every SCD is a fixed point for $\eta_b$, since it just reflects chains across the middle.

\section{Empirical data and AlphaEvolve}\label{computation}

In this section we present data to support Conjectures \ref{lconj} and \ref{mconj} generated by AlphaEvolve \cite{novikov2025alphaevolve}, an evolutionary LLM-based coding agent. AlphaEvolve has an evolutionary loop where the user defines a score of a given solution (which is a python program) and then AlphaEvolve asks an LLM to generate the next round of candidate programs with the goal of improving the score. In particular, it has recent success in pure mathematics in papers like \cite{georgiev2025mathematical}. We should mention the similarity to \cite{ellenberg2026bruhat} which also used AlphaEvolve to prove a new result about the Bruhat order. 

The goal of this section is not to create the best lower bounds, but to use AlphaEvolve to explore the space of algorithms for constructing and enumerating SCDs and try to learn from it. In fact, \cite{englandthesis2026} gives better bounds for $\#\SCD(L(m,n))$ than AlphaEvolve did. However, we gave AlphaEvolve very little computational power, only allowed it to run for at most two minutes and had it write programs in python, which is a notoriously inefficient language.  We note that for $M(n)$ with $n>7$ and $L(m,n)$ with $m+n>7$, AlphaEvolve hit computational limits for generating lower bounds, which makes it appear as though the number of SCD may be decreasing. As discussed in the introduction, one of the main takeaways is that AlphaEvolve was unable to find any consistent local rules for generating SCDs, but frequently used global properties or approximations of the representation-theoretic approaches as heuristics. 

Data was generated by describing the integer partition definition of $L(m,n)$ or $M(n)$, some equivalent definitions and properties (like crystal graphs, Bruhat orders of parabolic Weyl subgroups) in a prompt and then asking AlphaEvolve to maximize the number of unique SCDs it could construct for each example. The scoring function rigorously checked that the SCDs were valid and distinct. This gives lower bounds for $\#\SCD(M(n))$ and $\#\SCD(L(m,n))$, listed in Figures \ref{mnlb} and \ref{lmnlb} respectively. The full prompts and generated python code can be found at 

\href{github.com/google-research/symmetric-chain-decomposition}{\url{https://github.com/google-research/symmetric-chain-decomposition}}.

Like most reinforcement learning problems, AlphaEvolve seemed sensitive to the scoring function. It performed better when given a ``distance" from a correct solution, rather than a Boolean score of correct or incorrect for a valid SCD. We tried many different scoring functions. For finding a single SCD in large examples like $M(12)$, if the answer was not a valid SCD, the score was negative and the magnitude of the score was computed by how many partitions were in multiple or no chains, how far from symmetric the chains were and whether they were saturated. A positive score was only given for correct constructions. This setup allowed the system to hill climb on making it's solution more ``SCD-like" with subsequent attempts. We also had better success if the positive magnitude of the score for a correct solution was greater than the negative magnitude of an incorrect solution. Without this setup, it tended to get stuck in local minima of working only for small examples.

\begin{figure}[h]
\centering
\begin{tabular}{|l||c|c|c|c|c|c|c|}
\hline
    $n$ & 1 & 2 & 3 & 4 & 5 & 6 & 7 \\
    \hline\hline
    lower bound & 1 & 1 & 2 & 2 & 12 & $\geq$ 1344 & $\geq$ 360036 \\
    \hline
\end{tabular}
\caption{\label{mnlb} Lower bounds for $\#\SCD(M(n))$.}
\end{figure}
We can see that $\#\SCD(L(m,n)) = \#\SCD(L(n,m))$ by taking the conjugate of each partition. Therefore we only need to consider $m\geq n$. 
\begin{figure}[h]
\centering
\begin{tabular}{|c||c|c|c|c|c|c|}
\hline
    $n\backslash m$ & 1 & 2 & 3 & 4 & 5 & 6 \\
    \hline\hline
     1 & 1 & 1 & 1 & 1 & 1 & 1\\
    \hline
    2 &  & 2 & 2 & 6 & 6 & 24\\
    \hline
    3 &  &  & 12 & $\geq$ 396 & $\geq$ 15273 & $\geq$ 20000\\
    \hline
    4 &  &  &  & $\geq$ 34378 & $\geq$ 26930 & $\geq$ 15318 \\
    \hline
    5 &  &  &  &  & $\geq$ 5192 & $\geq$ 4856 \\
    \hline
    6 &  &  &  &  &  & $\geq$ 1478 \\
    \hline
\end{tabular}
\caption{\label{lmnlb} Lower bounds for $\#\SCD(L(m,n))$.}
\end{figure}

 \subsection{Algorithmic approaches of AlphaEvolve}

 In this section, we describe algorithmic approaches taken by AlphaEvolve. At a high level, AlphaEvolve treats the problem as a pathfinding problem with constraints. It finds valid routes through the poset and uses heuristics to avoid creating chains that violate the global constraints discussed in the other sections of the paper. None of these approaches were prompted for. 

 \subsubsection{Finding a single SCD}

 For finding a single SCD for a given set of $L(m,n)$ posets, AlphaEvolve favored bipartite matching approaches, often using network flow via custom implementations of Ford-Fulkerson or Edmonds-Karp. This approach was discussed in Ford and Fulkerson's original manuscript on network flow \cite{fordfulkerson} and Griggs \cite{griggs1977sufficient} later showed that if the poset has the ``LYM" property, then this method guarantees a symmetric chain decomposition. Griggs also showed that the posets we consider don't have this property in general. The lack of LYM property is also why the methods of \cite{tomon2025number} do not apply to our situation. Using network flow to find SCDs was also discussed in \cite{zjevik2014symmetric}.
 
 A brief description of the algorithm on an arbitrary graded poset $P$ is as follows. It uses network flow to grow chains upwards and downwards simultaneously and symmetrically. First, for simplicity, assume the number of ranks is even. There are modifications that can be made (and AlphaEvolve does make) if the number is odd, but the general idea is the same. The algorithm starts by considering the middle ranks and defines a network flow between them with two extra sources or sinks. All elements between the two middle ranks have edges of capacity 1 corresponding to the edges in the Hasse diagram. The lower middle rank is connected to the source and the upper middle rank is connected to the sink. Network flow is run to produce a matching between the two middle ranks. Next, we can recursively describe a process where we take the previously created chains between ranks $i$ and $\rho(P) - i$ and add edges with capacity constraints of 1 corresponding to the covering relations of the chains. Then extend the network to ranks $i-1$ and $\rho(P) - i + 1$, adding edges of capacity 1 that correspond to the covering relations and a new source and sink connected to ranks $i-1$ and $\rho(P) - i + 1$ respectively. Then run network flow and use the key insight that integer-valued capacities give integer-valued flow to turn the flow in to a Boolean valued function that tells which edges to take or discard for the chains. Any edges taken in the new network flow correspond to extending symmetric chains and any nodes that aren't taken from the source or sink are chains that end at that rank.

 For larger and more complicated searches with $M(n)$, the network flow described would time out and so AlphaEvolve reverted to heuristic bipartite matching approaches with backtracking, using custom implementations of Hopcroft-Karp. It converted the elements of $M(n)$ into bitmasks similar to the representation studied in \cite{tasoulas2024hamiltonian} and generated covering relations of $M(n)$ through bit manipulation for computational efficiency. Then, it created several heuristics to guide the bipartite matching, using different linear combinations of the heuristics to try different sortings for the bipartite matching via a heuristic portfolio technique. The heuristics use the following information.
 \begin{enumerate}
     \item The global in and out-degrees of elements in the current and next layer.
     \item The in degrees of the next layer, restricted to the nodes that are not terminating at the current rank. This can be computed via the starting rank of the chain the current node is on.
     \item The starting rank of the current chain for the elements in the current layer.
     \item The in degrees of elements in the next layer that are already taken by the partially constructed chains.
     \item The total number of paths downwards from elements in the current layer to the bottom and the total number of paths upwards from elements in the next layer to the top.
     \item The raw bitmasks of the target layer itself.
 \end{enumerate} 
 
 Many of these heuristics find analogs in Proctor's work. The $X$ and $Y$ operators are defined in terms of local covering relations, which can be simplified to in and out degrees. Successive applications of the $X$ operator counts the number of paths from an element to another element higher up in the Hasse diagram. The $Y$ operator is a weighted counting of the number of paths such that the sum of the weights along any path between two nodes is the same.

\subsubsection{Enumerating SCDs}

The AlphaEvolve code for enumerating SCDs uses many of the same heuristics as the previous section, with some extra tricks. When finding SCDs for $M(n)$, AlphaEvolve used pseudorandom searches of bipartite matches of adjacent ranks with the backtracking approach described in the previous section. Without being prompted, it exploited Proposition \ref{nofixed}, representing $\eta_s$ via bitwise complement to get a second SCD every time it found one. The author tried penalized it for using the {\bf random} module in python in an attempt to get it to find some deterministic rules, so it decided to implement a pseudorandom algorithm instead. It ran the backtracking algorithm for a fixed number of steps before abandoning it and moving on to another sorting to save time. When asked to optimize just for $M(7)$, the program introduced a number of advanced computational techniques via bytearrays and Marsaglia's Xorshift algorithm \cite{marsaglia2003xorshift} for pseudorandomness to reduce computational load. However, no new algorithmic approaches were found in that execution of AlphaEvolve.

When run on $L(m,n)$ it sometimes tried the same approach as it did for $M(n)$. In these cases, the author allowed it to use the random module. However, another approach taken for specific values of $L(m,n)$ was a randomized depth-first search approach, building chains in order of length, greedily trying to match nodes that were the most constrained first. A third approach was to generate as many partial SCDs from the bottom to the middle, reflect them over the middle using $\eta_b$, and then glue them together to get full SCDs. In this approach, it penalized highly used edges from previous SCDs to try to find new ones. Lastly, it combined Ford-Fulkerson as described in the last section with randomized orderings of the ranks.

\section{An implication of the conjectures}\label{implications}

Recently it was shown in \cite{orellana2024quasi, gutierrez2024towards} that finding certain SCDs for $L(m,n)$ would solve the plethysm problem of finding multiplicities $a_{1^n[r]}^k$ fitting in $$\Lambda^n\Sym^r\C^2 = \bigoplus_k(\Sym^k\C^2)^{\oplus a_{1^n[r]}^k}.$$
Stanley \cite{stanley2000positivity} listed this as one of the central problems in algebraic combinatorics and a special case was recently solved via a different method in \cite{pak2025combinatorial}. The conjectures imply that the search space for these constrained SCDs grows very quickly, which may suggest why finding an SCD that meets the criteria necessary for the plethysm problem has proven difficult. Computing plethysm coefficients was shown to be $\# P$-hard in \cite{fischer2020computational}.

For the poset $M(n)$, the elements are strict integer partitions, which classically index the irreducible polynomial representations of the queer Lie superalgebra $\mathfrak{q}(N)$, as well as the projective representations of the symmetric group. Note that the strict partitions indexing irreducible polynomial representations of $\mathfrak{q}(N)$ are categorized based on length restrictions, rather than largest part restrictions. However, just as maximal chains in $L(m,n)$ govern the combinatorial expansion standard Schur functions, the chains within $M(n)$ govern the combinatorial expansion of skew Schur $Q$-functions and encode the branching rules for these projective representations. If Conjecture \ref{mconj} holds, the super-exponential number of SCDs for $M(n)$ implies a large number of valid, symmetric combinatorial decompositions for the graded characters of these superalgebras.

\section{Conjectures, questions and open problems}\label{questions}

The main conjectures are Conjectures \ref{lconj} and \ref{mconj} from the introduction, which are a strengthening of Stanley's original open problem. The main perspective shift of this paper is that looking at the structure of the global set of SCDs for a given poset may be more fruitful than trying to find an algorithm that produces a single SCD.

\begin{question}
What is a more precise formulation of the asymptotics of the growth rates of Conjecture \ref{lconj} and \ref{mconj}?
\end{question}
As mentioned in the introduction, \cite{tomon2025number} recently gave asymptotic answers for the Boolean lattice and the hypergrid. Both of these posets have well known constructions of SCDs in contrast to the case we consider where existence is still an open problem. However, it would be interesting to see if any of the methods from that paper could be adapted to not rely on the LYM property and give asymptotics for our case.

Although the main conjectures concern growth rates, many of the concepts we introduced could benefit from further investigation.
\begin{question}
    Are there other nice bijective characterizations of $\#\SCD(L(m,n))$ or $\#\SCD(M(n))$ like Theorem \ref{l2n}?
\end{question}
Theorem \ref{l2n} gave a nice characterization in terms of permutations and inversion sets. The 12 SCDs for $M(5)$ can be thought of as a choice of an element of $S_2 \times S_3$ for the lengths of chains that pass through $\{21, 3\}$ and $\{32, 41, 5\}$, which are the ranks where it gains an element. This does not appear to work for $M(6)$ because not all equivalent ranks allow that independent choice and certain choices don't correspond to symmetric chain decompositions. Perhaps there is a characterization of the symmetries of the corresponding Weyl group that explains how to get around this? We list all SCDs for $M(3)$, $M(4)$ and $M(5)$ in Appendix \ref{app} for the interested reader.

\begin{question}
    Is there a combinatorial or algorithmic characterization of $\#\SCD$ in general?
\end{question}

Although we had a nice exact formula for $\#\SCD(L(2,n))$, finding something that simple in general is not an easy task. For small examples, we can also characterize SCDs by defining a ``tableaux avoidance" on their corresponding skew tableaux sequences.

Given a tableau $T$, recall that $R(T,i)$ is the shape of the Young diagram formed by deleting all boxes filled with numbers greater than $i$. We will now define a new skew SYT $R(T,i,j)$ of shape $R(T,j) / R(T,i)$, where each box of the Young diagram is filled with the corresponding entry of $T$ after subtracting $i$ from it. For example if 
$$ T = \begin{ytableau}
\redx & \redx & 1 & 2 & 3 & 4  \\
5 & 6& 7 & 8\\
\end{ytableau}$$

then 

$$R(T, 3, 6) = \begin{ytableau}
\redx & \redx & \redx & \redx & \redx & 1  \\
2 & 3\\
\end{ytableau}.$$

We say that $T$ {\it contains} another tableau $T'$ if $R(T,i,j) = T'$ for some $i,j$. We say that $T$ {\it avoids} $T'$ if $R(T,i,j) \neq T'$ for all $i,j$. We can think of tableaux avoidance of characterizing certain chains that cannot part of a SCD due to ``orphaning" another set of partitions. We will sometimes refer to $T'$ as a {\it pattern}.

The following is easily verifiable via brute force. 
\begin{proposition}\label{M4}
    A strict skew tableaux sequence $(T_i)$ of shapes $\lambda_i/\mu_i$ is strict admissible for $M(4)$ if and only if all of the following are met:
    \begin{enumerate} 
    \item We have $\lambda_i + \mu_i = \binom{n+1}{2}$ for all $i$,
    \item For all $\omega\in M(4)$, $R(T_i, j, k) = \omega$ for some $i,j,k$,
    \item All $T_i$ avoid:
        \begin{enumerate}
        \item $
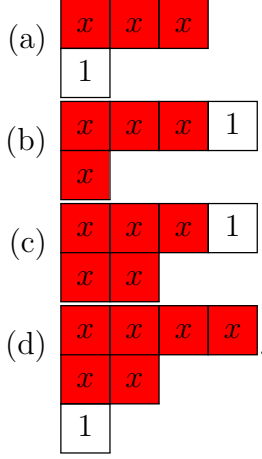
\begin{ytableau}
               \redx & \redx & \redx  \\
                1\\
                \end{ytableau}$
        \item $\begin{ytableau}
               \redx & \redx & \redx & 1  \\
                \redx\\
                \end{ytableau}$
        \item $\begin{ytableau}
               \redx & \redx & \redx & 1  \\
                \redx & \redx\\
                \end{ytableau}$
        \item $\begin{ytableau}
               \redx & \redx & \redx & \redx  \\
                \redx & \redx\\
                1\\
                \end{ytableau}$.
    \end{enumerate}
    \end{enumerate}
\end{proposition}

We can use tableaux avoidance to create a greedy algorithm for constructing SCDs. Construct the chains in order from largest to smallest, starting each new chain at an element that is at the smallest rank that currently has elements not part of the SCD. Greedily take edges allowed by the tableaux avoidance and also don't take elements already taken by previous chains. For $M(4)$, this always gives a valid SCD using the tableaux avoidance in Proposition \ref{M4} and gives both possible chains. So, for small examples like $M(4)$, the tableaux avoidance fully characterizes the possible SCDs. It is clear that tableaux avoidance is a necessary condition for SCDs. It would be very surprising if it was a sufficient condition in general, but it may be the case that this framework can be extended further to fully characterize the SCDs.

We can also see that in Proposition \ref{M4}, $(a) = \eta_s((d))$ and $(b) = \eta_s((c))$, suggesting that there may be some symmetry properties of tableaux avoidance itself that can be explored. For example, we can see that if taking a certain chain causes an element to be ``orphaned," then flipping that to the other side of the poset via the Lusztig involution would orphan the corresponding element on the other side.

Even if tableaux avoidance can fully characterize $\#\SCD$, there is a concern that the number of patterns needed to do this may also grow exponentially. This seems likely given that the plethysm problem is in $\# P$-hard.

\begin{question}
    Can the symmetry properties like $\eta_s$ and $\eta_b$ be used to construct SCDs or prove their existence?
\end{question}

Although Corollary \ref{coro} shows that $\#\SCD(M(n))$ is even, this unfortunately does not prove existence, as zero is an even number. Even if finding formulas for $\#\SCD$ remains out of reach, it's possible that the symmetries discussed in this paper could yield new ideas for constructing SCDs and solving Stanley's original problem. It could even be the case that describing the symmetries provides a nonconstructive proof for the existence of these SCDs in general, which itself is still open. Stanley's construction provides injective maps from $P_{i}$ to $P_{i+1}$ for $i$ below the middle of the poset and surjective maps above, but ``chaining" these injective maps together does not guarantee that the resulting chains are symmetric.

AlphaEvolve used $\eta_b$ to generate SCDs for $L(m,n)$ by generating valid chains up to the middle rank and then reflecting them over the middle rank via the involution and checking if the required edges existed to connect them. This could potentially be used for a structured approach to generating SCDs, but there are still some hurdles to overcome. For large examples, it's easy to create sets of chains that orphan elements without even crossing the middle rank.

\begin{question}
What is the structure of the other minuscule lattices?
\end{question}
We only considered the two main infinite families in the scope of this paper, but it would be interesting to explore what structure can be found in the exceptional cases with Cartan types $E_6$ and $E_7$. It is important to note that, with one exception, the minuscule lattices of Cartan types $C$ and $D$ are all isomorphic (as posets) to either $M(n)$ or $L(1,n)$, which is a single chain. The other infinite family is of Cartan type $D_n$ for the minuscule weight $\varpi_1$.

\begin{question}
    Can connections with representation theory, algebraic geometry, mathematical physics, or other areas of mathematics be used to make progress on the conjectures?
\end{question}

In this section, we discuss how the structure of SCDs of the minuscule lattices $L(m,n)$ and $M(n)$ arises from the structure of objects in representation theory, geometry and mathematical physics. We do not claim an exhaustive proof of these mappings, but rather present them as a research program intended to guide future inquiries into the existence and enumeration of symmetric chain decompositions.

Minuscule representations and their corresponding lattices indexed by minuscule weights and coweights appear in many parts of the literature.

By Weyl's complete reducibility, we know that if $V$ is a minuscule representation of $\mathfrak{g}$, then $$V \cong \bigoplus_{k=1}^{K} \mathcal{S}_k,$$ where each $\mathcal{S}_k$ is an $\mathfrak{sl}_2$ string. When looking at the construction of Proctor's operators in \cite{proctor1984bruhat}, he uses the principal $\mathfrak{sl}_2$ subalgebra of \cite{kostant1959principal}. The principal nilpotent operator, which is the sum of the simple root operators $e = \sum_{\alpha} e_\alpha$, acts on basis vectors of the representation by shifting the vector up by a linear combination of simple roots, which corresponds to a combination of edges in a minuscule crystal. Gross \cite{gross2000minuscule} shows that this corresponds to Weyl's decomposition. The principal $\mathfrak{sl}_2$ is in contrast to the entirely combinatorial decomposition given by an SCD $$P = \bigsqcup_{k=1}^K C_k.$$ An immediate corollary of Proctor's work is as follows.

\begin{proposition}\label{principalprop}
    Let $$V \cong \bigoplus_{k=1}^{K} \mathcal{S}_k$$ be the decomposition of a minuscule representation $V$ into irreducible $\mathfrak{sl}_2$ submodules. Then a symmetric chain decomposition of the corresponding minuscule lattice $$P = \bigsqcup_{j=1}^{w(P)} C_j$$ must have the following properties:
    \begin{enumerate}
        \item $w(P) = K$.
        \item There exists a bijection between the set of generators of height $i$ for irreducible submodules and the smallest elements $\mu$ of chains $C_j$ where $|\mu| = i$.
        \item The number of modules $S_k$ of dimension $d$ is equal to the number of chains $C_j$ with $d$ elements.
    \end{enumerate}
\end{proposition}

Gross goes further by showing that the principal $\mathfrak{sl}_2$ approach of Proctor and the Hard Lefschetz approach are connected by passing to the Langlands dual group. So we have the following analog of Proposition \ref{principalprop}.

\begin{proposition}\label{lefschetzprop}
    Let $$H^*(G/P, \mathbb{Q})\cong \bigoplus_{k=0}^{K-1} V_k$$ be the decomposition of a minuscule flag variety by the Lefschetz $\mathfrak{sl}_2$ operator into submodules generated by primitive classes with Langlands group $G$. Then a symmetric chain decomposition of the corresponding minuscule lattice $$P = \bigsqcup_{j=1}^{w(P)} C_j$$ for the Lie algebra of the dual Langlands group $G^{\vee}$ must have the following properties:
    \begin{enumerate}
        \item $w(P) = K$.
        \item There exists a bijection between the set of primitive classes of degree $i$ generating submodules and the smallest elements $\mu$ of chains $C_j$ where $|\mu| = i$.
        \item The number of submodules $V_k$ of dimension $d$ is equal to the number of chains $C_j$ with $d$ elements.
    \end{enumerate}
\end{proposition}

Therefore, the decomposition into the primitive cohomology that Stanley studied is exactly the decomposition of the Langlands dual in the representation and crystal approach. The smallest elements of the chains generate a basis for the principal $\mathfrak{sl}_2$ decomposition and a basis of primitive classes for the Lefschetz decomposition of the corresponding Langlands dual. By studying the maps from the set of SCDs to these two bases, we may be able to take a categorical approach to address Conjectures \ref{lconj} and \ref{mconj}.

We can expand on this further by considering the Affine Grassmannian. The minuscule lattices we study are indexed by the finite orbits of the Affine Grassmannian. The Affine Grassmannian is the central object in the Geometric Satake Equivalence \cite{mirkovic2007geometric}, which is a backbone of the Langlands Program. 

We can use this to create a bridge to physics. In \cite{costello2024q}, minuscule coweights, which index minuscule lattices, are used to construct 't Hooft lines, which are then used to construct Baxter $Q$-operators of the Gaudin model, the $q \rightarrow 1$ limit of the $Q$-operator of the XXZ quantum spin chain. In \cite{etingof2023general}, they show in certain cases of the analytic Langlands conjecture that these Baxter $Q$-operators are equivalent to the Hecke operators they study. Furthermore, they demonstrate that the joint spectrum of these Hecke operators is described by the Bethe Ansatz of the Gaudin model (parameterized by opers). The structure of these opers, and the associated commuting Hamiltonians, relies on weights derived from the decomposition of the Lie algebra via its principal $\mathfrak{sl}_2$ subalgebra, which ties back to Proposition \ref{principalprop}.

Coming back from the affine case to the finite case, in \cite{crichigno2024quantum}, the authors show a link between the Heisenberg XX spin chains and the minuscule lattices we study. They show that an element of $L(k, n-k)$ is equivalent to a classical configuration of a spin-$\frac{1}{2}$ chain with $n$ sites and magnon number $k$. The edges in the Hasse diagram of $L(k, n-k)$ correspond to local spin-exchange interactions between the spin chains. Furthermore, they show that quantized Schur functions are diagonalized by the Bethe ansatz for the XX spin chain, corresponding to the decomposition of the corresponding representation by the principal $\mathfrak{sl}_2$ subalgebra. 

How does this connect to SCDs? In the cases we discussed, we can see that the principal $\mathfrak{sl}_2$ decomposition (or equivalently the Lefschetz decomposition) are algebraic analogs of SCDs in the $q\rightarrow 1$ limiting case of quantum groups. In contrast, the true SCDs operate on the minuscule crystals, which are the $q\rightarrow 0$ limiting case of quantum groups. Therefore, the conjectures we raise in this paper concern the interplay between decompositions in these two limiting cases.

Chains in minuscule lattices have further geometric interpretations. The maximal chains in type $A$ (equivalent to standard Young tableaux) directly index the irreducible components of the corresponding Springer fibers. Via the Jacobson-Morozov theorem, the nilpotent element that defines the Springer fiber embeds into an $\mathfrak{sl}_2$ triple, acting as a raising operator \cite{collingwood1993nilpotent, chriss1997representation}. The existence of an SCD on a minuscule lattice $P$ is equivalent to the existence of a basis of the Springer fiber cohomology that is compatible with the $\mathfrak{sl}_2$-structure derived from the Jacobson-Morozov triple $(e, h, f)$.

Kashiwara and Saito \cite{kashiwara1997geometric} showed that the crystal graphs can be constructed using the irreducible components of the Lagrangian subvarieties of Nakajima quiver varieties. Later it was shown in \cite{maulik2019quantum} that the $R$-matrices and Hamiltonians of the quantum integrable spin chains discussed above can be constructed using the equivariant cohomology of Nakajima quiver varieties.

By shifting to studying these global structural relationships, we propose an avenue for progress on Stanley's problem.

\begin{question}
Which posets that satisfy the Hard Lefschetz theorem have a large number of SCDs?
\end{question}

Recently, the work of Huh and collaborators (for example \cite{adiprasito2018hodge}; see also the nice survey \cite{baker2018hodge}) has brought many ideas of algebraic geometry to combinatorics. Much of this work focuses on creating a K\"ahler package for certain combinatorial objects, which contains an analog of the Hard Lefschetz theorem even though there is no corresponding geometric interpretation. This can be seen as an extension of Stanley's original paper. In \cite{schweitzer2026existence}, the authors use symmetric chain decompositions of products of chains to construct their K\"ahler package, showing the reverse of Stanley's approach can also be useful. These products of chains are equivalent to the hypergrid for which \cite{tomon2025number} proved the number of SCDs grows super-exponentially. Is this true of all posets that satisfy this structure? Do all posets that satisfy this structure have at least one SCD?

\begin{problem}
Construct a quantum algorithm for finding an SCD for $L(m,n)$.
\end{problem}

Quantum computing was the original motivation behind the paper \cite{crichigno2024quantum}, where they gave quantum algorithms for computing particular coefficients like Littlewood-Richardson coefficients. Can their approach shed light on quantum algorithms for constructing a SCD? Informally, we can think of iterative applications of Proctor's $X$ operator as a ``superposition" of all possible paths from one element to elements above it and then interpret his result as showing SCDs exist in this ``superposition" state. If this is formalized via the connection to the Bethe Ansatz and the Fomin-Greene operators used in the paper, then it provides a potential path to a quantum algorithm.

In summary, moving the perspective from local to global structure, and connecting this global structure to representation theory, physics, or geometry, offers a promising path forward for making progress on Stanley's problem.

\bibliography{main}{}
\bibliographystyle{alpha} 

\appendix 

\section{Example Strict Skew Tableaux Sequences}\label{app}

We give all SCDs in their strict skew tableaux sequence format for $M(3), M(4), M(5)$. Recall from Section \ref{box} that if $T$ is of shape $\lambda/\mu$, then we denote $\mu$ via boxes filled with red x's. By including all valid SCDs, we hope the interested reader may look for patterns within the global solution space. 

\def\arraystretch{2.5}

M(3)

\begin{tabular}{ccc}
1 &
$
\begin{ytableau}
1 & 2 & 4 \\
3 & 5 \\
6 \\
\end{ytableau}
$ &
$\begin{ytableau}
\redx & \redx & \redx\\
\end{ytableau}$ \\
%%%%%%%%%%%%%%%%%%%%%%%%%%%%%%%%%%%%%%%%%%
2 &
$
\begin{ytableau}
1 & 2 & 3  \\
4 & 5 \\
6 \\
\end{ytableau}
$ &
$
\begin{ytableau}
\redx & \redx \\
\redx \\
\end{ytableau}
$ \\

\end{tabular}

---------------------------------------------------------------

M(4)

\begin{tabular}{ccc}
1 &
$
\begin{ytableau}
1 & 2 & 4 & 7 \\
3 & 5 & 8 \\
6 & 9 \\
10
\end{ytableau}
$ &
$\begin{ytableau}
\redx & \redx & \redx & 1 \\
2 & 3 & 4 \\
\end{ytableau}$ \\
%%%%%%%%%%%%%%%%%%%%%%%%%%%%%%%%%%%%%%%%%%
2 &
$
\begin{ytableau}
1 & 2 & 3 & 4 \\
5 & 6 & 7 \\
8 & 9 \\
10
\end{ytableau}
$ &
$
\begin{ytableau}
\redx & \redx & 1 & 4 \\
\redx & 2 \\
3  \\
\end{ytableau}
$ \\

\end{tabular}

---------------------------------------------------------------

M(5)

\begin{tabular}{cccc}
1 &
$
\begin{ytableau}
1 & 2 & 3 & 4 & 5 \\
6 & 7 & 8 & 9 \\
10 & 11 & 12 \\
13 & 14 \\
15 \\
\end{ytableau}
$ &
$\begin{ytableau}
\redx & \redx & 1 & 4 & 5 \\
\redx & 2 & 6 & 9  \\
3 & 7  \\
8 \\
\end{ytableau}$ 
&
$\begin{ytableau}
\redx & \redx & \redx & \redx \\
\redx & 1 & 2 \\
3 & 4\\
5\\
\end{ytableau}$ 

\\

%%%%%%%%%%%%%%%%%%%%%%%%%%%%%%%%%%%%%%%%%%
2 &
$
\begin{ytableau}
1 & 2 & 3 & 4 & 5 \\
6 & 7 & 8 & 9 \\
10 & 11 & 12 \\
13 & 14 \\
15 \\
\end{ytableau}
$ &
$\begin{ytableau}
\redx & \redx & 1 & 2 & 8 \\
\redx & 3 & 4 & 9  \\
5 & 6 \\
7
\end{ytableau}$ 
&
$\begin{ytableau}
\redx & \redx & \redx & 2 & 3 \\
\redx & \redx & 4 \\
1 & 5
\end{ytableau}$ 
\\
%%%%%%%%%%%%%%%%%%%%%%%%%%%%%%%%%%%%%%%%%%
3 &
$
\begin{ytableau}
1 & 2 & 4 & 7 & 11 \\
3 & 5 & 8 & 12 \\
6 & 9 & 13 \\
10 & 14 \\
15 \\
\end{ytableau}
$ &
$\begin{ytableau}
\redx & \redx & \redx & 1 & 5 \\
2 & 3 & 4 & 6 \\
7 & 8 & 9 \\
\end{ytableau}$ 
&
$\begin{ytableau}
\redx & \redx & \redx & \redx & \redx \\
1 & 2 & 4 \\
3 & 5 \\
\end{ytableau}$ 
\\
%%%%%%%%%%%%%%%%%%%%%%%%%%%%%%%%%%%%%%%%%%
4 &
$
\begin{ytableau}
1 & 2 & 4 & 7 & 11 \\
3 & 5 & 8 & 12 \\
6 & 9 & 13 \\
10 & 14 \\
15 \\
\end{ytableau}
$ &
$\begin{ytableau}
\redx & \redx & \redx & 1 & 2 \\
3 & 4 & 6 & 8 \\
5 & 7 & 9 \\
\end{ytableau}$ 
&
$\begin{ytableau}
\redx & \redx & \redx & \redx & 3  \\
\redx & 1 & 2 & 4  \\
5\\
\end{ytableau}$ 
\\
%%%%%%%%%%%%%%%%%%%%%%%%%%%%%%%%%%%%%%%%%%
5 &
$
\begin{ytableau}
1 & 2 & 3 & 4 & 8 \\
5 & 6 & 7 & 9 \\
10 & 11 & 12 \\
13 & 14 \\
15 \\
\end{ytableau}
$ &
$\begin{ytableau}
\redx & \redx & 1 & 4 & 8  \\
\redx & 2 & 5 & 9  \\
3 & 6 \\
7 \\
\end{ytableau}$ 
&
$\begin{ytableau}
\redx & \redx & \redx & \redx & \redx  \\
1 & 2 & 4  \\
3 & 5\\
\end{ytableau}$ 
\\
%%%%%%%%%%%%%%%%%%%%%%%%%%%%%%%%%%%%%%%%%%
6 &
$
\begin{ytableau}
1 & 2 & 3 & 4 & 11 \\
5 & 6 & 7 & 12 \\
8 & 9 & 13 \\
10 & 14 \\
15 \\
\end{ytableau}
$ &
$\begin{ytableau}
\redx & \redx & 1 & 4 & 5  \\
\redx & 2 & 6 & 8  \\
3 & 7 & 9 \\
\end{ytableau}$ 
&
$\begin{ytableau}
\redx & \redx & \redx & \redx & \redx  \\
1 & 2 & 3 & 4  \\
5\\
\end{ytableau}$ 
\\
\end{tabular}

\begin{tabular}{cccc}
%%%%%%%%%%%%%%%%%%%%%%%%%%%%%%%%%%%%%%%%%%
7 &
$
\begin{ytableau}
1 & 2 & 4 & 5 &8\\
3 & 6& 7 & 9\\
10 & 11 & 12\\
13 & 14\\
15\\
\end{ytableau}
$ &
$\begin{ytableau}
\redx & \redx & \redx & 1 & 2  \\
3 & 4 & 6 & 9\\
5 & 7\\
8\\
\end{ytableau}$ 
&
$\begin{ytableau}
\redx & \redx & \redx & 2   \\
\redx & \redx & 3  \\
1 & 4\\
5
\end{ytableau}$ 
\\
%%%%%%%%%%%%%%%%%%%%%%%%%%%%%%%%%%%%%%%%%%
8 &
$
\begin{ytableau}
1 & 2 & 4 & 5 & 11\\
3 & 6 & 7 & 12\\
8 & 9 & 13\\
10 & 14\\
15\\
\end{ytableau}
$ &
$\begin{ytableau}
\redx & \redx & \redx & 1 & 2  \\
3 & 4 & 5 & 6\\
7 & 8 & 9\\
\end{ytableau}$ 
&
$\begin{ytableau}
\redx & \redx & \redx & 2 & 3  \\
\redx & \redx & 4  \\
1 & 5\\
\end{ytableau}$ 
\\
%%%%%%%%%%%%%%%%%%%%%%%%%%%%%%%%%%%%%%%%%%
9 &
$
\begin{ytableau}
1 & 2 & 3 & 4 & 5 \\
6 & 7 & 9 & 11 \\
8 & 10 & 12 \\
13 & 14 \\
15 \\
\end{ytableau}
$ &
$\begin{ytableau}
\redx & \redx &  1 & 4 & 8  \\
\redx & 2 & 5 & 9\\
3 & 6\\
7
\end{ytableau}$ 
&
$\begin{ytableau}
\redx & \redx & \redx & \redx & 3  \\
\redx & 1 & 2 & 4  \\
5\\
\end{ytableau}$ 
\\
%%%%%%%%%%%%%%%%%%%%%%%%%%%%%%%%%%%%%%%%%%
10 &
$
\begin{ytableau}
1 & 2 & 3 & 4 & 5 \\
6 & 7 & 9 & 12 \\
8 & 10 & 13 \\
11 & 14 \\
15 \\
\end{ytableau}
$ &
$\begin{ytableau}
\redx & \redx &  1 & 2 & 5  \\
\redx & 3 & 4 & 6 \\
7 & 8 & 9 \\
\end{ytableau}$ 
&
$\begin{ytableau}
\redx & \redx & \redx & 2  \\
\redx & \redx & 3  \\
1 & 4\\
5 \\
\end{ytableau}$ 
\\
%%%%%%%%%%%%%%%%%%%%%%%%%%%%%%%%%%%%%%%%%%
11 &
$
\begin{ytableau}
1 & 2 & 4 & 7 & 8\\
3 & 5 & 9 & 11\\
6 & 10 & 12\\
13 & 14\\
15\\
\end{ytableau}
$ &
$\begin{ytableau}
\redx & \redx &  \redx &  1 & 8\\
2 & 3 & 4 & 9 \\
5 & 6 \\
7 \\
\end{ytableau}$ 
&
$\begin{ytableau}
\redx & \redx & \redx & \redx & \redx  \\
1 & 2 & 3 & 4 \\
5 \\
\end{ytableau}$ 
\\
%%%%%%%%%%%%%%%%%%%%%%%%%%%%%%%%%%%%%%%%%%
12 &
$
\begin{ytableau}
1 & 2 & 4 & 7 & 8 \\
3 & 5 & 9 & 12 \\
6 & 10 & 13 \\
11 & 14 \\
15 \\
\end{ytableau}
$ &
$\begin{ytableau}
\redx & \redx &  \redx &  1 & 2\\
3 & 4 & 5 & 6\\
7 & 8 & 9\\
\end{ytableau}$ 
&
$\begin{ytableau}
\redx & \redx & \redx & \redx   \\
\redx & 1 & 2\\
3 & 4\\
5 \\
\end{ytableau}$ 
\\

\end{tabular}

\hspace{0.5cm}

\end{document}